\documentclass[11pt]{amsart}

\usepackage[T1]{fontenc}
\usepackage{amsmath,amssymb}
\usepackage{mathrsfs}
\usepackage{graphicx}
\usepackage{needspace}
\usepackage{xcolor}
\usepackage{url}
\usepackage{hyperref}

\definecolor{modra3}{rgb}{.1,.0,.4}
\hypersetup{colorlinks=true, linkcolor=modra3, urlcolor=modra3, citecolor=modra3, pdfpagemode=UseNone, pdfstartview=}

\usepackage{caption}
\renewcommand{\descriptionlabel}[1]%
  {{\hglue -0.7 cm}\hspace{\labelsep}#1}
\newcommand*{\Scale}[2][4]{\scalebox{#1}{\ensuremath{#2}}}
\newcommand\PP{{\mathscr P}}
\newcommand\FF{{\mathscr F}}
\newcommand{\PPi}{\mathbf{\Xi}}
\newcommand{\kode}[3]{{#1}_{#2}^{#3}}
\newcommand{\LL}{\mathscr L}
\newcommand{\CC}{\mathscr C}
\newcommand{\redit}[1]{\textcolor{red}{#1}}
\newcommand{\blueit}[1]{\textcolor{blue}{#1}}
\newcommand{\bpl}{{\Scale[3]{+}}}

\newtheorem{theorem}{Theorem}
\newtheorem{proposition}[theorem]{Proposition}
\newtheorem{observation}[theorem]{Observation}
\newtheorem{lemma}[theorem]{Lemma}

\theoremstyle{remark}
\newtheorem*{remark}{Remark}

\begin{document}

\title[$k$-arrangements of pseudolines and pseudocircles]{$k$-arrangements of pseudolines\\ and pseudocircles}

\author[Kyn\v{c}l]{Jan Kyn\v{c}l}
\address{Department of Applied Mathematics, Faculty of Mathematics and Physics, Charles University, Prague, Czech Republic}
\email{kyncl@kam.mff.cuni.cz}

\author[Medina]{Carolina Medina}
\address{Instituto de Ciencias B\'asicas e Ingenier\'\i a, Universidad Aut\'onoma del Estado de Hidalgo, Pachuca, Mexico.}
\email{carolina\_medina@uaeh.edu.mx}

\author[Salazar]{Gelasio Salazar}
\address{Instituto de F\'\i sica, Universidad Aut\'onoma de San Luis Potos\'{\i}, Mexico.}
\email{gsalazar@ifisica.uaslp.mx}


\begin{abstract}
A {\em $k$-arrangement of pseudolines} is a set of bi-infinite curves in the plane such that any two of them intersect each other in exactly $k$ points, at which they cross, and it is {\em simple} if no three curves meet at a common point. Cyclic arrangements are the only simple $1$-arrangements of pseudolines that are {\em unavoidable}, in the Ramsey spirit: for each fixed $m\ge 1$, every sufficiently large simple $1$-arrangement of pseudolines has a cyclic subarrangement of size $m$. We show that, for every $m\ge 3$, the number of unavoidable simple $k$-arrangements of pseudolines of size $m$ grows exponentially with $k$, independently of $m$. For even $k$, we prove an analogous result for $k$-arrangements of pseudocircles.
\end{abstract}

\maketitle

\section{Introduction}\label{sec:intro}

An {\em arrangement of pseudolines} is a set of bi-infinite curves in the plane such that any two of them intersect at exactly one point, at which they cross. An {\em arrangement of pseudocircles} is a set of simple closed curves in the plane such that any two intersect at exactly two points, at which they cross. 

An arrangement of pseudolines or pseudocircles is {\em simple} if no three curves intersect at a common point. Simple arrangements of pseudolines are of special interest because they provide a geometric model for uniform oriented matroids of rank $3$, as explained for instance in the survey by Felsner and Goodman~\cite{f0}.

Throughout this work, all arrangements under consideration are simple. If $A$ is an arrangement, a subset of $A$ is a {\em subarrangement} of $A$.

As in the standard reference by Matou\v{s}ek~\cite{matousek}, we adopt the notion that two arrangements of pseudolines are {\em isomorphic} if there is a self-homeomor\-phism of the plane that takes one into the other. Using the same criterion for isomorphism between arrangements of pseudocircles is unnecessarily restrictive, and we adopt the notion used by Felsner and Scheucher~\cite{Felsner2}: two arrangements of pseudocircles are {\em isomorphic} if they induce homeomorphic cell decompositions of the compactified plane, that is, of the sphere. Throughout this work we make no distinction between isomorphic arrangements (of pseudolines or pseudocircles).

\subsection{Unavoidable arrangements}

On the left hand side of Figure~\ref{fig:fig1} we illustrate the {\em cyclic} arrangement of $5$ pseudolines. The construction is naturally generalized to obtain the cyclic arrangement of $m$ pseudolines, for any positive integer $m$. In the spirit of the work by Pach, Solymosi, and T{\'o}th~\cite{Pach1}, cyclic arrangements are the {\em unavoidable} arrangements of pseudolines: for each positive integer $m$, every sufficiently large arrangement of pseudolines contains a cyclic subarrangement of size $m$. This was established by Ram\'\i rez-Alfons\'\i n~\cite{Alfonsin1}. We also note that from a purely combinatorial perspective this property is an immediate consequence of the Folkman-Lawrence Topological Representation Theorem. We refer the reader to the classic text by Bj\"orner, Las Vergnas, Sturmfels, White, and Ziegler~\cite[Chapter]{bjorner} for the underlying chirotope and oriented matroid machinery that guarantees such unavoidable substructures.

In the same vein, Medina, Ram\'\i rez-Alfons\'{\i}n and Salazar proved~\cite{pams} that for each integer $m\ge3$, there are three unavoidable arrangements of pseudocircles of size $m$, in the sense that every sufficiently large arrangement of pseudocircles contains one of these as a subarrangement. The unavoidable arrangements of pseudocircles of size $5$ are shown in Figure~\ref{fig:fig1}. Another result in this direction is the quantitative version of the Erd\H{o}s--Szekeres theorem for lines proved by B\'ar\'any, Rold\'an-Pensado, and T\'oth~\cite{barany1}.

\begin{figure}[ht!]
\centering
\resizebox{\textwidth}{!}{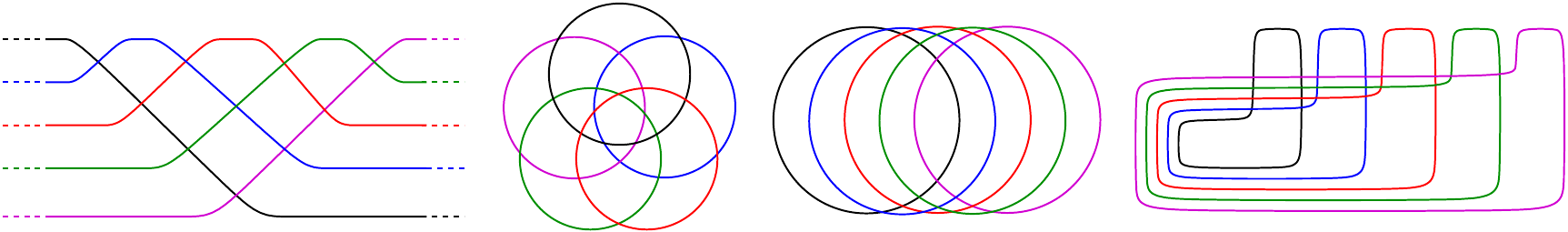}
\caption{The cyclic arrangement of size $5$ (left) is the only unavoidable arrangement of $5$ pseudolines. We also illustrate the three unavoidable arrangements of $5$ pseudocircles.}
\label{fig:fig1}
\end{figure}

\subsection{$k$-arrangements and the motivation behind this work}

Arrangements of pseudolines and pseudocircles are naturally generalized to admit the possibility of more crossings between the curves. Let $k$ be a positive integer. A $k$-{\em arrangement of pseudolines} is a collection of pseudolines that pairwise intersect in exactly $k$ points, at which they cross. If $k\ge 2$ is an even integer, then a $k$-{\em arrangement of pseudocircles} is a collection of pseudocircles that pairwise intersect in exactly $k$ points, at which they cross. Under this generalized notion, an arrangement of pseudolines, as defined above, is a $1$-arrangement of pseudolines, and an arrangement of pseudocircles is a $2$-arrangement of pseudocircles. 

This work was motivated by two questions raised by a reviewer of the paper~\cite{pams} (see~\cite[Section 5]{pams}): 

\begin{description}
\item[(I)] Which are the unavoidable $k$-arrangements of pseudocircles? 
\item[(II)] How does the number of unavoidable $k$-arrangements of pseudocircles increase with $k$? 
\end{description}

As we discuss in Section~\ref{sub:pacmea}, a full answer to (I) seems to be beyond hope at this moment. Our main result involves (II): we show that both for $k$-arrangements of pseudolines and pseudocircles the growth is exponential (see Theorem~\ref{thm:main} below).

\subsection{Unavoidable $k$-arrangements and packable $k$-arrangements}

Before any further discussion we must clarify the notion of an unavoidable collection of arrangements. Let $k\ge 1$ and $m\ge 2$ be integers. A collection $\LL_{k,m}$ of $k$-arrangements of pseudolines of size $m$ is {\em unavoidable} if it satisfies the following properties:

\begin{description}
\item[(U1)] there is an integer $n_0$ such that for each $n\ge n_0$, every $k$-arrangement of size $n$ has a subarrangement isomorphic to a $k$-arrangement in $\LL_{k,m}$; and
\item[(U2)] $\LL_{k,m}$ is minimal with respect to (U1).
\end{description}

The notion of an unavoidable collection $\CC_{k,m}$ of $k$-arrangements of pseudocircles of size $m$ is totally analogous.

It may not seem immediately obvious that for any $k\ge 1$ and $m\ge 2$ the collections $\LL_{k,m}$ and $\CC_{k,m}$ (if $k$ is even) are unique. As we shall now see, we establish their uniqueness from the equivalence of unavoidability with the notion of {\em packability}.

Let $A$ and $B$ be $k$-arrangements (of pseudolines or pseudocircles) of sizes $m$ and $n$, respectively, with $n\ge m$. We say that $B$ is $A$-{\em packed} if every subarrangement of $B$ of size $m$ is isomorphic to $A$. If for each $p\ge m$ there exists a $k$-arrangement of size $p$ that is $A$-packed, then $A$ is {\em packable}.

\begin{proposition}\label{pro:unavpack}
Let $k\ge1$ and $m\ge2$ be integers. Then $\LL_{k,m}$ is unique: it consists precisely of those $k$-arrangements of $m$ pseudolines that are packable. If $k$ is even, the same holds for $\CC_{k,m}$ and pseudocircles.
\end{proposition}

For instance, $\LL_{1,5}$ consists only of the arrangement of size $5$ illustrated in Figure~\ref{fig:fig1}, and $\CC_{2,5}$ consists of the three $2$-arrangements of pseudocircles also in Figure~\ref{fig:fig1}.

\begin{remark} 
For the rest of this paper we will not use again the term unavoidable, working instead under the equivalent notion of packable arrangements: $\LL_{k,m}$ (respectively, $\CC_{k,m}$) is the collection of those $k$-arrangements of pseudolines (respectively, pseudocircles) of size $m$ that are packable.
\end{remark}

\subsection{The main result}

We answer Question (II): the sizes of $\LL_{k,m}$ and $\CC_{k,m}$ increase exponentially with $k$, independently of $m$. 

\begin{theorem}\label{thm:main} 
There are absolute constants $a,b$ with $a{>}1$, such that for all sufficiently large $k$ the following holds. For all integers $m\ge 3$, 
\[
a^k < |\LL_{k,m}| < b^k, \text{\rm \hglue 0.3 cm and if $k$ is even\hglue 0.3 cm} a^k < |\CC_{k,m}| < b^k.
\]
\end{theorem}

The lower bounds are proved in Section~\ref{sec:lower}, and the upper bounds are proved in Sections~\ref{sec:codes} and~\ref{sec:upperbound}.

We remark that as a by-product of another project, Dobbins, Holmsen, and Hubard~\cite{DHH} characterized the packable $k$-arrangements of $x$-monotone pseudolines. In particular, their number grows exponentially with $k$, and so this in particular implies the lower bounds in Theorem~\ref{thm:main}. We briefly review this part of their work in Section~\ref{sub:regpat}. 

Aiming to have a paper as self-contained as possible, in Section~\ref{sec:lower} we give an elementary argument that proves the lower bounds in the theorem. Thus the core of the paper is devoted to the upper bounds. After some preliminary results on codes given in Section~\ref{sec:codes}, the proof of the upper bounds is given in Section~\ref{sec:upperbound}.

\subsection{Previous work on \texorpdfstring{$k$}{k}-arrangements}

Arrangements of pseudolines or pseudocircles that pairwise cross {\em at most} $k$ times (a slight variant of our notion) have been investigated by Chan~\cite{chan1,chan2,chan3}, by Edelsbrunner, Guibas, Pach, Pollack, Seidel, and Sharir~\cite{ede}, by Kyn{\v{c}}l,  Pach, Radoi{\v{c}}i{\'c}, and T{\'o}th~\cite{kprt}, by Ortner~\cite{ortner} and by Snoeyink and Hershberger~\cite{swee}.

The problem of realizing prescribed incidences between finite sets of abstract points and abstract lines by points and straight lines dates back to the nineteenth century. As explained for instance in the seminal treatise of Gr\"unbaum~\cite{gru3}, not every prescribed incidence structure can be realized by points and straight lines. Allowing arrangements of pseudolines gives more possibilities, but still does not suffice: the incidences of the Fano plane, for instance, cannot be realized even in this way. On the other hand, Bokowski, Kovi\v{c}, Pisanski, and \v{Z}itnik~\cite{bokowski2} proved that every finite incidence structure can be realized if one allows each pair of pseudolines to cross finitely many times instead of exactly once.

\subsection{Proof of Proposition~\ref{pro:unavpack}}

We prove the proposition for arrangements of pseudolines, as the argument for arrangements of pseudocircles is identical.

\begin{proof}[Proof of Proposition~\ref{pro:unavpack} for pseudoline arrangements]
Let $\PP_{k,m}$ denote the collection of packable $k$-arrangements of size $m$. We show that (i) $\PP_{k,m}$ is contained in every collection of $k$-arrangements of size $m$ that satisfies (U1); and that (ii) $\PP_{k,m}$ satisfies (U1). 

To see that (i) and (ii) imply the proposition we note that it follows from (i) and (ii) that $\PP_{k,m}$ satisfies (U2). Moreover, any collection satisfying (U1) and (U2) contains $\PP_{k,m}$ (by (i)), and so by minimality it must be exactly equal to $\PP_{k,m}$, and so the proposition follows.

For the rest of the proof for brevity we refer to a $k$-arrangement of pseudolines simply as an {\em arrangement}.

To prove (i) let $\FF_{k,m}$ be any collection of arrangements of size $m$ that satisfies (U1) with threshold $n_{\FF}$. Let $A\in \PP_{k,m}$. An $A$-packed arrangement of size $\max\{m,n_\FF\}$ has a subarrangement isomorphic to an element of $\FF_{k,m}$, and each of its $m$-subarrangements is isomorphic to $A$. Therefore necessarily $A$ is in $\FF_{k,m}$, and so $\PP_{k,m}\subseteq \FF_{k,m}$. Thus (i) follows.

To prove (ii) we first note that if $A$ is an arrangement of size $m$ that is not in $\PP_{k,m}$ then there is an $N_A$ such that no $A$-packed arrangement of size $N_A$ exists, and therefore none of any size larger than $N_A$, since every subarrangement of an $A$-packed arrangement is $A$-packed. Since there is a finite number of isomorphism classes of arrangements of size $m$, we may let $N=\max\{m,\,N_A\}$, where the maximum is taken over all isomorphism classes $A$ of arrangements of size $m$ that are not in $\PP_{k,m}$.

Now let $q$ be the number of isomorphism classes of arrangements of size $m$, and let $n_0$ be a Ramsey number such that every $q$-colouring of the $m$-subsets of an $n_0$-set has a monochromatic subset of size $N$. 

Let $B$ be an arrangement of size $n\ge n_0$, and colour each $m$-subset by its isomorphism class. 

By Ramsey's theorem $B$ has a monochromatic subset $S$ of size $N$. Thus $B[S]$ is an arrangement of size $N$ that is $C$-packed for some arrangement $C$. The choice of $N$ implies that $C$ is packable, since otherwise $N_C\le N$ and no $C$-packed arrangement of size $N$ would exist. Therefore $B$ contains a packable arrangement of size $m$, that is, an arrangement in $\PP_{k,m}$. Thus $\PP_{k,m}$ satisfies (U1), as claimed in (ii).
\end{proof}

\section{Proof of the lower bounds in Theorem~\ref{thm:main}}\label{sec:lower}

In the proof of the lower bounds of Theorem~\ref{thm:main} we use the following standard notation and terminology for words. Suppose that $W$ is a word $a_1 a_2\cdots a_\ell$ on a two-letter alphabet $\{\alpha,\beta\}$. We use $\overline{W}$ to denote the word obtained by performing a {\em complementary} operation on $W$, that is, interchanging $\alpha$ and $\beta$. We use $W^{-1}$ to denote the word obtained by performing a {\em reversal} operation on $W$, that is, $W^{-1}=a_\ell \cdots a_2 a_1$. Finally, for each $i=1,2,\ldots,\ell$ the {\em $i$th cyclic translation} of $W$ is the word $a_i a_{i+1} \cdots a_\ell a_1 \cdots a_{i-1}$.

\begin{proof}[Proof of the lower bounds in Theorem~\ref{thm:main}]
Let $m\ge 3$ and $k\ge 1$ be integers (with $k$ even if we consider $k$-arrangements of pseudocircles). We show that $|\LL_{k,m}|\ge 2^k/4$ and $|\CC_{k,m}| \ge 2^k/(4k)$. From this it follows that both lower bounds in Theorem~\ref{thm:main} hold for any $a$ such that $1 < a < 2$.

As we illustrate on the left-hand side of Figure~\ref{fig:st} for the case $r=5$, for each $r\ge 3$ we let $S_r$ be a bounded section of the cyclic arrangement of size $r$. Thus $S_r$ consists of $r$ $x$-monotone arcs that cross each other exactly once. As we show on the right-hand side of Figure~\ref{fig:st}, we let $T_r$ be the section of arcs obtained by reflecting $S_r$ over a horizontal line. 

\begin{figure}[ht!]
\def\ta#1{{\Scale[1.4]{#1}}}
\centering
\scalebox{0.7}{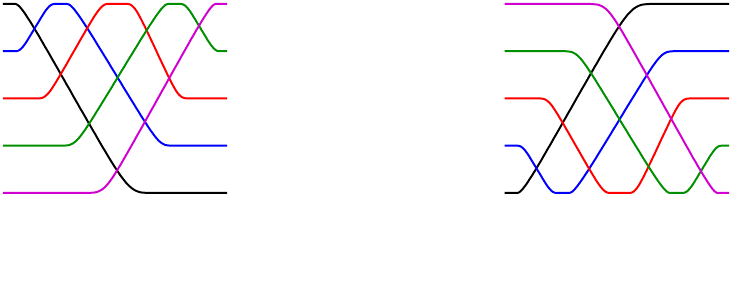}
\caption{$S_5$ is a bounded section of the cyclic arrangement of size $5$, and $T_5$ is obtained by reflecting $S_5$ over a horizontal line. The sections $S_r$ and $T_r$ are obtained similarly for each $r\ge 3$.}
\label{fig:st}
\end{figure}

As we illustrate in Figure~\ref{fig:ST}, we construct $k$-arrangements of $r$ pseudolines by gluing together $k$ copies of $S_r$ or $T_r$ and extending the result infinitely in both directions. In the terminology of Section~\ref{sec:concluding}, the objects obtained by gluing together copies of $S_r$ and $T_r$ are particularly simple instances of {\em regular path systems}~\cite{DHH}. 

\begin{figure}[ht!]
\def\ta#1{{\Scale[2.0]{#1}}}
\centering
\scalebox{0.5}{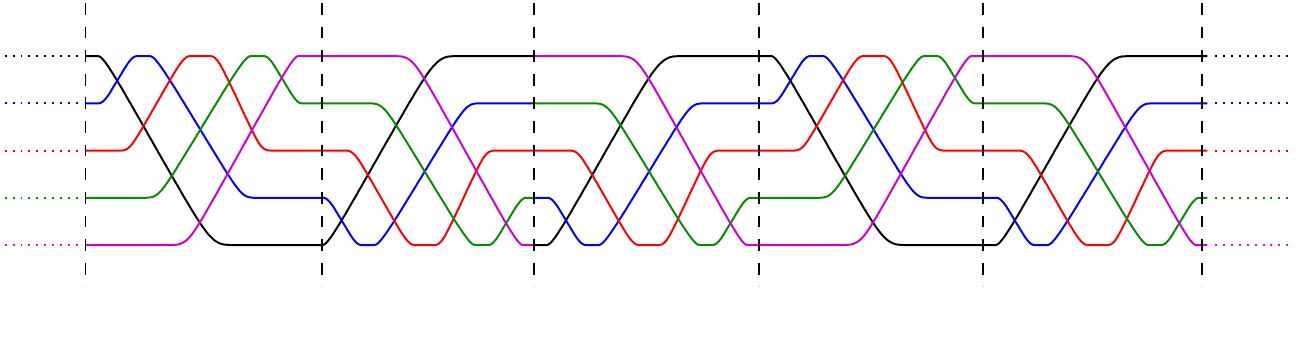}
\caption{This $5$-arrangement of $5$ pseudolines is obtained from the word $S_5 T_5 T_5 S_5 T_5$: we glue one copy of $S_5$, followed by two copies of $T_5$, then followed by one more copy of $S_5$ and finally by one copy of $T_5$, and extend the result infinitely in both directions.} 
\label{fig:ST}
\end{figure}

To formalize this construction let $W$ be a word of length $k$ on the alphabet $\{S_r,T_r\}$. By concatenating the corresponding $S_r$- and $T_r$-sections and extending to infinity in both directions we obtain a $k$-arrangement of pseudolines $A_r(W)$. For every $n\ge m$, every $m$-subarrangement of $A_n(W)$ is isomorphic to $A_m(W)$, and so it follows that $A_m(W)$ is packable.

To see how many elements of $\LL_{k,m}$ (that is, non-isomorphic packable $k$-arrangements of size $m$) we obtain in this way, fix $k$ and $m$. It is straightforward to see that if $W$ and $U$ are two words of length $k$ then the $k$-arrangements $A_m(W)$ and $A_m(U)$ are isomorphic if and only if $U\in \{W,\overline{W},$ $({W})^{-1},(\overline{W})^{-1}\}$ (we note that $(\overline{W})^{-1}=\overline{W^{-1}}$). Since there are $2^k$ words of length $k$ it follows that there are at least $2^k/4$ non-isomorphic packable $k$-arrangements of pseudolines of size $m$. That is, $|\LL_{k,m}|\ge 2^k/4$. This proves the lower bound in Theorem~\ref{thm:main} for $\LL_{k,m}$.

For $k$-arrangements of pseudocircles given a word of length $k$ on the alphabet $\{S_r, T_r\}$ we glue together the $S_r$- and $T_r$-sections but instead of extending the glued object to infinity in both directions we join the loose ends, as illustrated in Figure~\ref{fig:ST2}, in order to obtain a $k$-arrangement of $r$ pseudocircles $B_r(W)$. For every $n\ge m$, every $m$-subarrangement of $B_n(W)$ is isomorphic to $B_m(W)$, and so it follows that $B_m(W)$ is packable.

\begin{figure}[ht!]
\def\ta#1{{\Scale[2.2]{#1}}}
\centering
\scalebox{0.45}{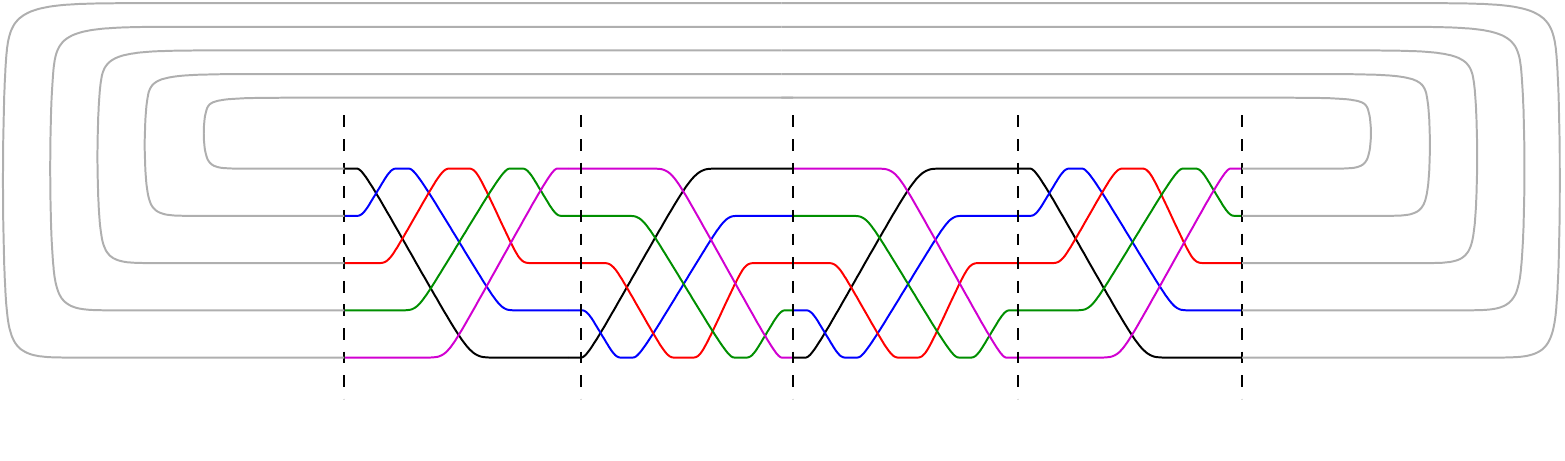}
\caption{This $4$-arrangement of $5$ pseudocircles is obtained from the word $S_5 T_5 T_5 S_5$: we glue one copy of $S_5$, followed by two copies of $T_5$, and finally followed by one copy of $S_5$, and join the loose ends with pairwise disjoint arcs.}
\label{fig:ST2}
\end{figure}

It is straightforward to see that if $W$ and $U$ are two words of length $k$ then the $k$-arrangements of pseudocircles $B_m(W)$ and $B_m(U)$ are isomorphic if and only if $U$ is obtained from $W$ by a sequence of complementary operations, reversal operations, and cyclic translations. It is easily checked that given a word $W$ of length $k$ one can obtain at most $4k$ different words by performing these operations, and from this it follows that there are at least $2^k/(4k)$ non-isomorphic packable $k$-arrangements of pseudocircles of size $m$. That is, $|\CC_{k,m}|\ge 2^k/(4k)$. This proves the lower bound in Theorem~\ref{thm:main} for $\CC_{k,m}$.
\end{proof}

\section{Codes of \texorpdfstring{$k$}{k}-arrangements of pseudocircles}\label{sec:codes}

To prove the upper bound in Theorem~\ref{thm:main} for $k$-arrangements of pseudocircles we turn the problem into a combinatorial one. We capture a $k$-arrangement in a combinatorial {\em code} that retains the order and identity of the crossings along each pseudocircle. As we shall see, the key fact is that a code of size at least $3$ is determined by its subcodes of size $3$. In Section~\ref{sec:upperbound}, a single application of Ramsey's theorem will reduce the counting problem for packable arrangements of any size to the counting of codes of size $3$.

\begin{remark} 
Throughout this section we only work with $k$-arrangements of pseudocircles, and for brevity we refer to them simply as $k$-arrangements. 
\end{remark} 

Our encoding is an extension of the one used by Ortner~\cite{ortner} for $2$-arrangements. An alternative way to encode $k$-arrangements is described in the related work by Linhart and Ortner~\cite{linort1}. For a positive integer $r$, we write $[r]=\{1,\ldots,r\}$.

\subsection{Codes of \texorpdfstring{$k$}{k}-arrangements}

Let $B$ be a $k$-arrangement of size $n\ge2$. Label its pseudocircles $1,\ldots,n$, and choose a noncrossing point $p_i$ (a {\em startpoint}) on each pseudocircle $i$. The encoding procedure is illustrated in Figure~\ref{fig:codespsc}. 

We traverse each pseudocircle counterclockwise from its startpoint. Suppose that a crossing $v$ of pseudocircles $i$ and $j$ is the $s$th crossing with $i$ encountered in the traversal of $j$. In the traversal of $i$, record $v$ as $j_s^+$ if we enter the closed disk bounded by $j$, and as $j_s^-$ if we leave this disk. The resulting sequence $\sigma(i)$ is the {\em key} of pseudocircle $i$, and the ordered tuple $\Sigma=(\sigma(1),\ldots,\sigma(n))$ is a {\em code} of $B$. As with $B$, we say that the code $\Sigma$ has {\em size} $n$.

\def\pseone{{ {\Scale[2.4]{1}}}}
\def\psetwo{{ \blueit{{\Scale[2.4]{2}}}}}
\def\psethree{{ \redit{{\Scale[2.4]{3}}}}}

\def\pseoneb{{ {\Scale[2.4]{\mathbf 1}}}}
\def\psetwob{{ \blueit{{\Scale[2.4]{\mathbf 2}}}}}
\def\psethreeb{{ \redit{{\Scale[2.4]{\mathbf 3}}}}}

\def\fxone{{ {\Scale[2.4]{p}}_{\Scale[1.8]{1}} }}
\def\fxtwo{{ {\Scale[2.4]{p}}_{\Scale[1.8]{2}} }}
\def\fxthree{{ {\Scale[2.4]{p}}_{\Scale[1.8]{3}} }}
\def\xone{ {\Scale[2.6]{3}}^{\Scale[2]{+}}_{\Scale[2]{3}} }
\def\xtwo{ {\Scale[2.6]{3}}^{\Scale[2]{-}}_{\Scale[2]{2}}}
\def\xthree{ {\Scale[2.6]{3}}^{\Scale[2]{+}}_{\Scale[2]{1}}}
\def\xfour{ {\Scale[2.6]{3}}^{\Scale[2]{-}}_{\Scale[2]{4}}}
\def\xfive{ {\Scale[2.6]{2}}^{\Scale[2]{+}}_{\Scale[2]{3}}}
\def\xsix{ {{\Scale[2.6]{2}}^{\Scale[2]{-}}_{\Scale[2]{4}}}}
\def\xseven{ {{\Scale[2.6]{2}}^{\Scale[2]{+}}_{\Scale[2]{1}}}}
\def\xeight{ {\Scale[2.6]{2}}^{\Scale[2]{-}}_{\Scale[2]{2}}}

\def\xnine{ {\Scale[2.6]{1}}^{\Scale[2]{+}}_{\Scale[2]{1}} }
\def\xten{ {\Scale[2.6]{3}}^{\Scale[2]{+}}_{\Scale[2]{4}}}
\def\xeleven{ {\Scale[2.6]{3}}^{\Scale[2]{-}}_{\Scale[2]{1}}}
\def\xtwelve{ {\Scale[2.6]{1}}^{\Scale[2]{-}}_{\Scale[2]{2}}}
\def\xthirteen{ {\Scale[2.6]{1}}^{\Scale[2]{+}}_{\Scale[2]{3}}}
\def\xfourteen{ {{\Scale[2.6]{3}}^{\Scale[2]{+}}_{\Scale[2]{2}}}}
\def\xfifteen{ {{\Scale[2.6]{3}}^{\Scale[2]{-}}_{\Scale[2]{3}}}}
\def\xsixteen{ {\Scale[2.6]{1}}^{\Scale[2]{-}}_{\Scale[2]{4}}}

\def\zone{ {\Scale[2.6]{1}}^{\Scale[2]{-}}_{\Scale[2]{1}} }
\def\ztwo{ {\Scale[2.6]{1}}^{\Scale[2]{+}}_{\Scale[2]{4}}}
\def\zthree{ {\Scale[2.6]{1}}^{\Scale[2]{-}}_{\Scale[2]{3}}}
\def\zfour{ {\Scale[2.6]{1}}^{\Scale[2]{+}}_{\Scale[2]{2}}}
\def\zfive{ {\Scale[2.6]{2}}^{\Scale[2]{+}}_{\Scale[2]{3}}}
\def\zsix{ {{\Scale[2.6]{2}}^{\Scale[2]{-}}_{\Scale[2]{4}}}}
\def\zseven{ {{\Scale[2.6]{2}}^{\Scale[2]{+}}_{\Scale[2]{1}}}}
\def\zeight{ {\Scale[2.6]{2}}^{\Scale[2]{-}}_{\Scale[2]{2}}}

\def\ga{{\Scale[2.6]{\text{\rm (a)}}}}
\def\gb{{\Scale[2.6]{\text{\rm (b)}}}}
\def\gc{{\Scale[2.6]{\text{\rm (c)}}}}
\def\gd{{\Scale[2.6]{\text{\rm (d)}}}}
\begin{figure}[tbp]
\centering
\resizebox{\textwidth}{!}{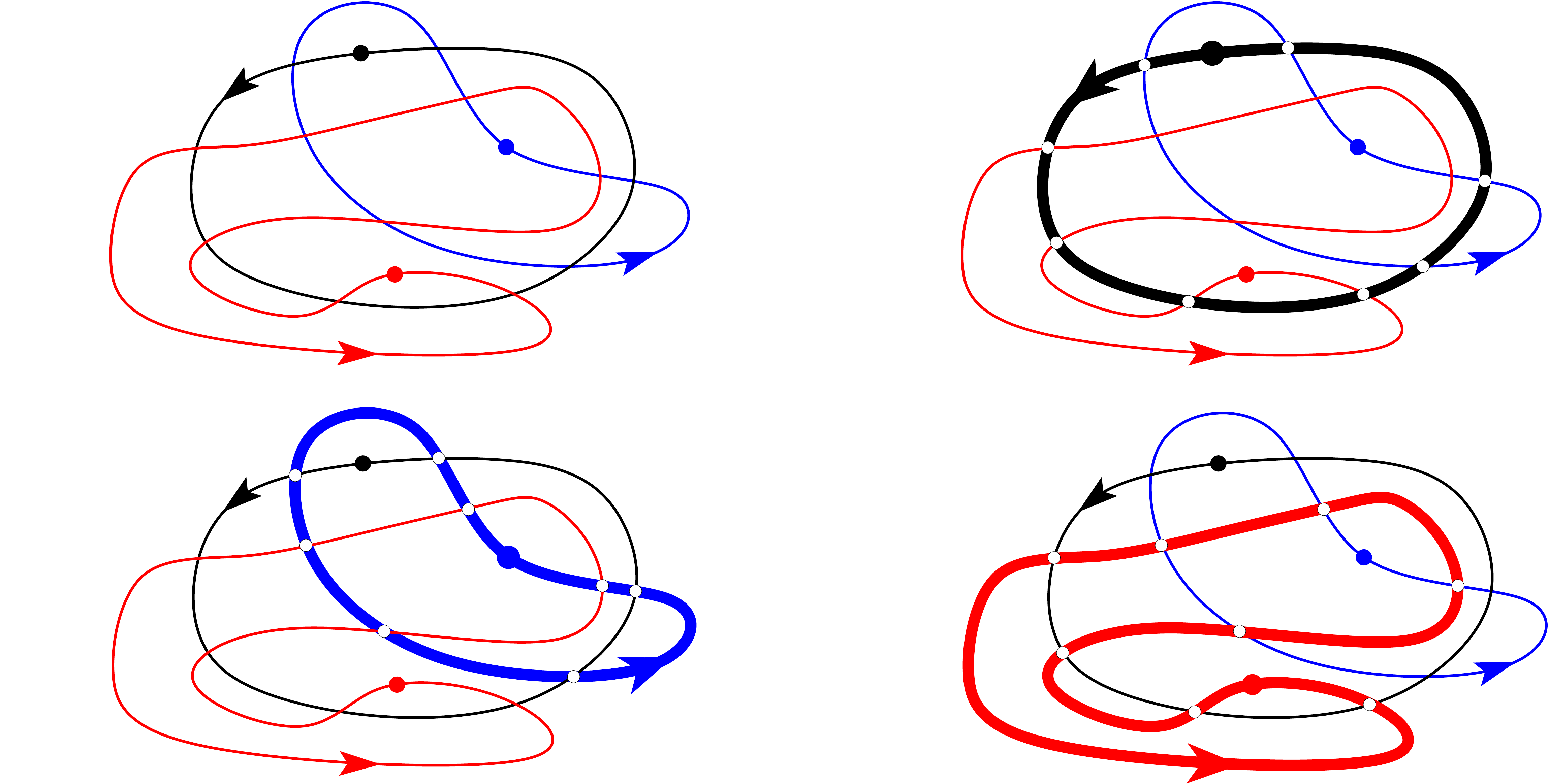}
\caption{A $4$-arrangement $B$ of three pseudocircles, with labels and startpoints as in (a). Panel (b) gives the key $\sigma(1)=\kode{2}{2}{-} \kode{3}{3}{+} \kode{3}{2}{-} \kode{3}{1}{+} \kode{3}{4}{-} \kode{2}{3}{+} \kode{2}{4}{-} \kode{2}{1}{+}$; panel (c) gives $\sigma(2)=\kode{3}{3}{-} \kode{1}{4}{-} \kode{1}{1}{+} \kode{3}{4}{+} \kode{3}{1}{-} \kode{1}{2}{-} \kode{1}{3}{+} \kode{3}{2}{+}$; and panel (d) gives $\sigma(3)=\kode{1}{3}{-} \kode{1}{2}{+} \kode{2}{3}{+} \kode{2}{4}{-} \kode{2}{1}{+} \kode{2}{2}{-} \kode{1}{1}{-} \kode{1}{4}{+}$. Thus $\Sigma=(\sigma(1),\sigma(2),\sigma(3))$ is a code of $B$.}
\label{fig:codespsc}
\end{figure}

Clearly each key $\sigma(i)$ has $(n-1)k$ symbols. For each $j\in[n]\setminus\{i\}$ and $s\in[k]$, exactly one of $j_s^+$ and $j_s^-$ occurs in $\sigma(i)$, and so the symbols in a key are all distinct. We denote by $\PPi_{k,n}$ the collection of all codes of $k$-arrangements of $n$ pseudocircles. We note that evidently $\PPi_{k,n}$ is finite.


We emphasize that a $k$-arrangement $B$ need not have a unique code: in general, different labellings of the pseudocircles, different choices of startpoints, and different choices of the outer face yield different codes. 

If we regard a $k$-arrangement $B$ as a $4$-regular graph $G_B$ embedded in the sphere then it is not difficult to see that the code of $B$ determines $G_B$ and its rotation system (see the standard reference by Mohar and Thomassen~\cite{MoharThomassen} for a review on rotation systems). Since $G_B$ is connected, its rotation system determines its embedding in the sphere up to homeomorphism. Since this is precisely the notion of isomorphism that we use for pseudocircle arrangements, this implies the following.

\begin{observation}\label{obs:3circ} Let $A$ and $B$ be $k$-arrangements of pseudocircles of the same size. If there is a $\Sigma$ that is a code of both $A$ and $B$, then $A$ and $B$ are isomorphic.
\end{observation}

\subsection{Subcodes}\label{sub:subcodes}

Let $\Sigma=(\sigma(1),\ldots,\sigma(n))$ be a code of an arrangement $B$, and let $S=\{i_1,\ldots,i_m\}\subseteq[n]$, where $i_1<\cdots<i_m$ and $m\ge2$. The {\em subcode} $\Sigma[i_1,\ldots,i_m]$, also denoted by $\Sigma[S]$, is obtained by retaining the keys indexed by $S$, deleting from these keys all symbols whose curve labels are outside $S$, and relabelling $i_j\mapsto j$ everywhere. See Figure~\ref{fig:coderest} for an illustration.

Thus $\Sigma[i_1,\ldots,i_m]=\Sigma[S]$ is the code of the corresponding subarrangement with the same startpoints and the inherited order of the curve labels. The signs and indices $s$ are unchanged. Indeed, each index records an order among crossings of one fixed pair of curves, and deleting other curves does not change that order. 

We note that deleting some pseudocircles and then deleting some more gives exactly the same code as deleting all of them at once, after the natural relabelling of the remaining pseudocircles. This implies that taking subcodes in stages has the same effect as taking the final subcode in one step from the original code.

\def\fxone{{ {\Scale[2.4]{p}}_{\Scale[1.8]{1}} }}
\def\fxtwo{{ {\Scale[2.4]{p}}_{\Scale[1.8]{2}} }}
\def\xone{ {\Scale[2.6]{2}}^{\Scale[2]{+}}_{\Scale[2]{3}} }
\def\xtwo{ {\Scale[2.6]{2}}^{\Scale[2]{-}}_{\Scale[2]{2}}}
\def\xthree{ {\Scale[2.6]{2}}^{\Scale[2]{+}}_{\Scale[2]{1}}}
\def\xfour{ {\Scale[2.6]{2}}^{\Scale[2]{-}}_{\Scale[2]{4}}}

\def\zone{ {\Scale[2.6]{1}}^{\Scale[2]{-}}_{\Scale[2]{1}} }
\def\ztwo{ {\Scale[2.6]{1}}^{\Scale[2]{+}}_{\Scale[2]{4}}}
\def\zthree{ {\Scale[2.6]{1}}^{\Scale[2]{-}}_{\Scale[2]{3}}}
\def\zfour{ {\Scale[2.6]{1}}^{\Scale[2]{+}}_{\Scale[2]{2}}}

\def\pseone{{ {\Scale[2.4]{1}}}}
\def\psetwo{{ \redit{{\Scale[2.4]{2}}}}}

\def\pseoneb{{ {\Scale[2.4]{\mathbf 1}}}}
\def\psetwob{{ \redit{{\Scale[2.4]{\mathbf 2}}}}}

\def\ga{{\Scale[2.6]{\text{\rm (a)}}}}
\def\gb{{\Scale[2.6]{\text{\rm (b)}}}}
\def\gc{{\Scale[2.6]{\text{\rm (c)}}}}
\begin{figure}[tbp]
\centering
\resizebox{\textwidth}{!}{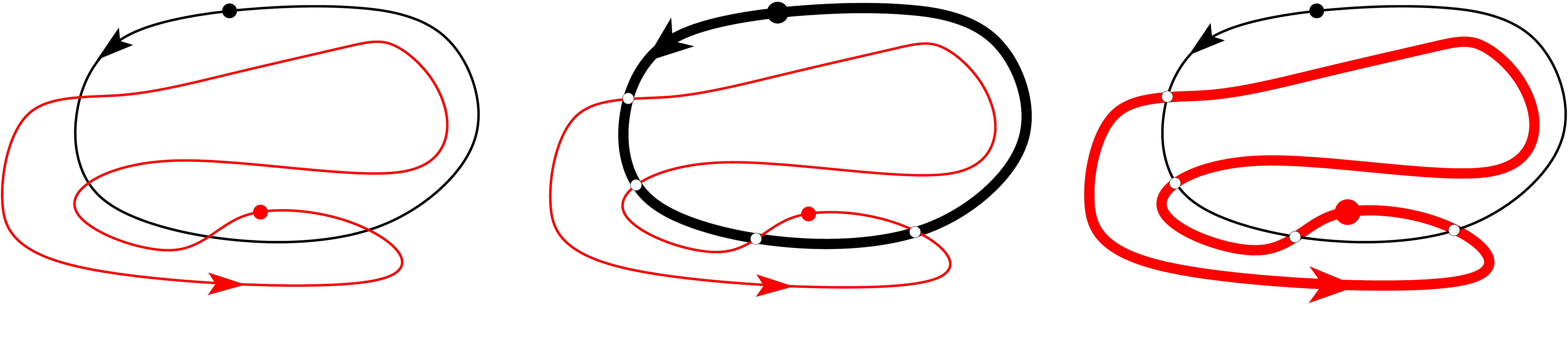}
\caption{Restricting the arrangement $B$ in Figure~\ref{fig:codespsc} to pseudocircles $1$ and $3$, and relabelling $3\mapsto2$, gives the arrangement $A$ in (a). Keeping the same startpoints, its keys are read in (b) and (c). Hence $\Sigma[1,3]=\bigl(\kode{2}{3}{+} \kode{2}{2}{-} \kode{2}{1}{+} \kode{2}{4}{-}\,,\,\kode{1}{3}{-} \kode{1}{2}{+} \kode{1}{1}{-} \kode{1}{4}{+}\bigr)$.}
\label{fig:coderest}
\end{figure}

\subsection{Reconstruction from triple subcodes}

We now show a key property of the subcodes of a code. Loosely speaking, we prove that the $3$-subcodes of a code contain enough information to determine the whole code.

\begin{lemma}\label{lem:triplesdo}
Let $m\ge3$, and let $\Sigma,\Phi\in\PPi_{k,m}$. If $\Sigma[S]=\Phi[S]$ for every $3$-element subset $S$ of $[m]$, then $\Sigma=\Phi$.
\end{lemma}

\begin{proof}
Write $\Sigma=(\sigma(1),\ldots,\sigma(m))$ and $\Phi=(\varphi(1),\ldots,\varphi(m))$. Every $2$-subcode is obtained by restricting any $3$-subcode containing the same pair, so corresponding $2$-subcodes are equal. Thus $\sigma(i)$ and $\varphi(i)$ contain the same symbols, with the same relative order whenever the symbols involve the same other curve.

Now consider two symbols $j_s^\delta$ and $\ell_t^\epsilon$ in $\sigma(i)$ with $j\ne\ell$. Their relative order is recorded in $\Sigma[\{i,j,\ell\}]$. Equality with $\Phi[\{i,j,\ell\}]$, with the same increasing relabelling on both sides, shows that their relative order in $\varphi(i)$ is the same. Thus every pair of symbols has the same relative order in the two keys. Since their symbols are distinct, $\sigma(i)=\varphi(i)$. This holds for every $i$, and hence $\Sigma=\Phi$.
\end{proof}

In particular, for a fixed $\Pi\in\PPi_{k,3}$ and a fixed $m\ge3$, there is at most one code of size $m$ all of whose $3$-subcodes equal $\Pi$.

\section[Proof of the upper bounds]{Proof of the upper bounds in Theorem~\ref{thm:main}}\label{sec:upperbound}

We prove the upper bounds separately, first for pseudocircles and then for pseudolines. The two proofs have the same basic form. We first show that in each case the number of codes of size $3$ is at most exponential in $k$. We then apply Ramsey's theorem to the $3$-subcodes of any sufficiently large packed arrangement. For pseudocircles, Lemma~\ref{lem:triplesdo} from Section~\ref{sec:codes} will then imply that the resulting common $3$-subcode determines the whole code. The pseudoline case is handled analogously.


Before moving on to the proofs we briefly review the notion of a meander, a classical object that plays a central role in some of our upcoming arguments.

\subsection{Meanders and $k$-arrangements of size two}

As we illustrate in Figure~\ref{fig:openmea}, an {\em open meander of order $k$} consists of a horizontal line together with a pseudoline that crosses it exactly $k$ times.

\begin{figure}[ht!]
\def\tf#1{{\Scale[2.8]{#1}}}
\centering
\scalebox{0.38}{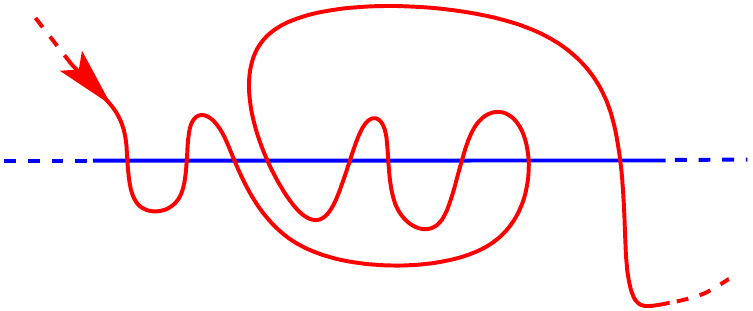}
\caption{An open meander of order $9$.}
\label{fig:openmea}
\end{figure}

As we illustrate in Figure~\ref{fig:closmea}, a {\em closed meander of order $k$} consists of a horizontal line together with a simple closed curve that crosses it exactly $2k$ times. 

\begin{figure}[ht!]
\def\tf#1{{\Scale[2.8]{#1}}}
\centering
\scalebox{0.38}{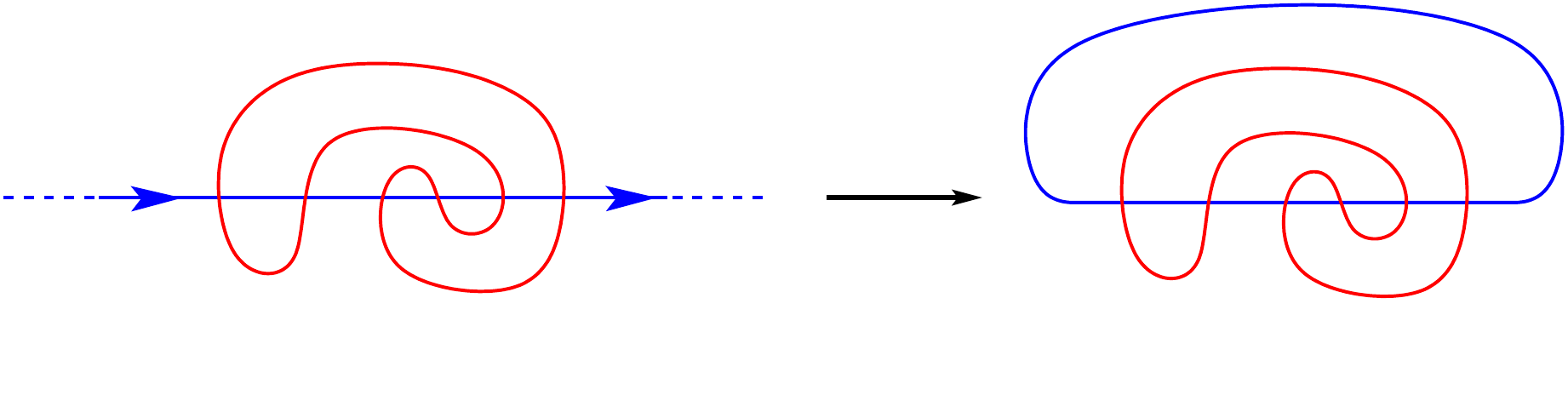}
\caption{In (a) we have a closed meander of order $3$. As we show in (b), this closed meander is naturally turned into a $6$-arrangement of two pseudocircles.}
\label{fig:closmea}
\end{figure}

Two meanders of the same type (both open or both closed) are {\em equivalent} if there is an orientation-preserving self-homeomorphism of the plane that takes one to the other and takes the horizontal line to itself, preserving its direction.

We use meanders to find upper bounds on the number of isomorphism classes of arrangements of size two.

For each positive integer $k$ let $L(k,2)$ denote the set of isomorphism classes of $k$-arrangements of two pseudolines, and for each even $k$ let $C(k,2)$ denote the set of isomorphism classes of $k$-arrangements of two pseudocircles.

\begin{lemma}\label{lem:meanbound}
There is an absolute constant $D>1$ such that $|L(k,2)|<D^k$ for every integer $k \ge 1$ and $|C(k,2)|<D^k$ for every even integer $k \ge 2$.
\end{lemma}

\begin{proof}
Straightening one of the two pseudolines of a $k$-arrangement into a horizontal line turns it into an open meander of order $k$, and equivalent meanders yield isomorphic arrangements; thus $|L(k,2)|$ is at most the number of open meanders of order $k$. Similarly, cutting one of the two pseudocircles of a $k$-arrangement at a noncrossing point and opening it into a horizontal line shows that $|C(k,2)|$ is at most the number of closed meanders of order $k/2$. 

Both the number of open meanders of order $r$ and the number of closed meanders of order $r$ grow at most exponentially with $r$. See the survey by La~Croix~\cite{lacroix}, and in particular the work by Albert and Paterson~\cite{albert} for the closed case. Thus the lemma follows.
\end{proof}

Albert and Paterson~\cite{albert} show that the exponential growth rate of closed meanders (counted by half their number of crossings) lies between $11.380$ and $12.901$. Thus the exponential growth rate per crossing lies between $\sqrt{11.380}\approx 3.3734$ and $\sqrt{12.901}\approx 3.5918$. For a simple exponential upper bound on the number of meanders, one may consider the method of Albert and Paterson~\cite[Section 2]{albert} that assigns to each closed meander with $k$ crossings a distinct word of length $k$ over a four-letter alphabet, and from this it follows that there are at most $4^k$ such meanders.

We will show that both upper bounds in Theorem~\ref{thm:main} hold with any fixed $b>64D^3$, where $D$ is as in Lemma~\ref{lem:meanbound}.


\subsection{Proof of the upper bound in Theorem~\ref{thm:main} for $\CC_{k,m}$}\label{sec:upcirc}

Recall that $\PPi_{k,n}$ is the collection of all codes of $k$-arrangements of $n$ pseudocircles. We start by establishing the following upper bound for $|\PPi_{k,3}|$.

\begin{lemma}\label{lem:third}
$|\PPi_{k,3}|<4096k^6(64D^3)^k$ holds for every even integer $k\ge 2$, where $D$ is as in Lemma~\ref{lem:meanbound}.
\end{lemma}

\begin{proof}
We first bound the number $|\PPi_{k,2}|$ of codes of size $2$. Consider one isomorphism class. To obtain a code we must first specify, for each of the two pseudocircles, which of its two sides will be the bounded side in the plane. There are at most $4$ possibilities for this. We then choose one of the two orientations of the sphere, one of the two labellings of the pseudocircles, and one of the $k$ possible crossing-free intervals for the startpoint of each pseudocircle. Thus each isomorphism class yields at most $4\cdot2\cdot2\cdot k^2=16k^2$ codes. Using  Lemma~\ref{lem:meanbound} we obtain
\begin{equation}\label{eq:eq7}
|\PPi_{k,2}|\le 16k^2\cdot D^k
\end{equation}
for every even integer $k\ge2$.

Now let $\Sigma$ be a code of size $3$. There are three $2$-subcodes, namely $\Sigma[1,2]$, $\Sigma[1,3]$, and $\Sigma[2,3]$, and so there are at most $|\PPi_{k,2}|^3$ possibilities for them. Once these three subcodes are fixed, consider one of the three keys of $\Sigma$. Its restrictions to the other two pseudocircles are two fixed strings, each of length $k$. The key is obtained by merging these two strings while preserving the order within each of them. There are at most $\binom{2k}{k}$ ways to do this, since it is enough to choose which $k$ of the $2k$ positions are occupied by the first string. Since there are three keys, it follows that
\begin{equation}\label{eq:tcc}
|\PPi_{k,3}|\le |\PPi_{k,2}|^3\cdot \binom{2k}{k}^{\!3}.
\end{equation}
The choices counted on the right need not all give valid codes, but this does not affect the validity of the inequality. Using $\binom{2k}{k}<4^k$ and~\eqref{eq:eq7}, it follows that 
\[
|\PPi_{k,3}|<(16k^2)^3 D^{3k}(4^k)^3=4096k^6(64D^3)^k.
\]
\end{proof}


\begin{proof}[Proof of the upper bound in Theorem~\ref{thm:main} for $|\CC_{k,m}|$]
Let $m\ge3$ and fix any $b>64D^3$. Since $4096k^6(64D^3)^k<b^k$ holds for all sufficiently large $k$, Lemma~\ref{lem:third} implies that to prove the upper bound it suffices to show that
\begin{equation}\label{eq:suff1}
|\CC_{k,m}|\le |\PPi_{k,3}|
\end{equation}
for every even integer $k \ge 2$.

Let $q:=|\PPi_{k,3}|$. By Ramsey's theorem there is an $n\ge m$ such that every colouring of the (unordered) triples of $\{1,\ldots,n\}$ with $q$ colours has an $m$-element subset all of whose triples have the same colour. Let $A\in\CC_{k,m}$. Since $A$ is packable, there is an $A$-packed arrangement $B$ of size $n$. Choose a code $\Omega$ of $B$, with any labelling and startpoints, and colour each triple $\{i,j,\ell\}$, $i<j<\ell$, by the induced subcode $\Omega[i,j,\ell]\in\PPi_{k,3}$.

Ramsey's theorem gives an $m$-element subset $S$ all of whose triples have the same colour. Denote this common colour, which is a code in $\PPi_{k,3}$, by $\Pi_A$. Let $A^*=B[S]$ and let $\Sigma_A=\Omega[S]$. Then $A^*$ is isomorphic to $A$ and $\Sigma_A$ is a code of $A^*$. Moreover, as we pointed out at the end of Section~\ref{sub:subcodes}, taking subcodes in stages has the same effect as taking the corresponding subcode in one step. Hence every $3$-subcode of $\Sigma_A$ is equal to $\Pi_A$.

In this way for each $A\in\CC_{k,m}$ we obtain a corresponding code $\Pi_A\in\PPi_{k,3}$. To prove~\eqref{eq:suff1}, it remains to show that two non-isomorphic arrangements in $\CC_{k,m}$ cannot give the same code $\Pi_A$.

To prove this suppose that $A'\in\CC_{k,m}$ and that $\Pi_A=\Pi_{A'}$. Applying the construction above to $A'$, we obtain an arrangement $(A')^*$ isomorphic to $A'$ and a code $\Sigma_{A'}$ of $(A')^*$ with the property that every $3$-subcode of $\Sigma_{A'}$ is equal to $\Pi_{A'}$. Therefore every $3$-subcode of $\Sigma_A$ is equal to the corresponding $3$-subcode of $\Sigma_{A'}$. By Lemma~\ref{lem:triplesdo} this implies that $\Sigma_A=\Sigma_{A'}$, and by Observation~\ref{obs:3circ} it follows that $A^*$ and $(A')^*$ are isomorphic. Since $A^*$ is isomorphic to $A$ and $(A')^*$ is isomorphic to $A'$, it follows that $A$ and $A'$ are isomorphic. Thus two non-isomorphic arrangements in $\CC_{k,m}$ give distinct codes in $\PPi_{k,3}$, and~\eqref{eq:suff1} follows. 
\end{proof}

\subsection{Proof of the upper bound in Theorem~\ref{thm:main} for $\LL_{k,m}$}\label{sec:uplin}

We finish the proof of Theorem~\ref{thm:main} with the upper bound for $|\LL_{k,m}|$. Codes for pseudoline arrangements are defined as in Section~\ref{sec:codes}, except that on each pseudoline we choose a direction instead of a startpoint. More precisely, label the pseudolines $1,\ldots,n$ and choose a direction on each. Suppose that a crossing of $i$ and $j$ is the $s$th crossing with $i$ encountered when traversing $j$. In the key of $i$, record this crossing as $j_s^+$ if $i$ crosses $j$ from the right side to the left side of the oriented pseudoline $j$, and as $j_s^-$ otherwise. Reading the crossings along $i$ in its chosen direction gives the key of $i$, and the ordered collection of all the keys is a \emph{linear code}. We denote by $\PPi^{\mathrm L}_{k,n}$ the collection of all linear codes of $k$-arrangements of $n$ pseudolines. 

Subcodes are defined by deletion, exactly as before, and as in Section~\ref{sub:subcodes} taking subcodes in stages has the same effect as taking the corresponding subcode in one step. A linear code also determines its arrangement up to isomorphism. To see this we take a closed disk containing all crossings and meeting each pseudoline exactly twice. We discard the part outside the disk and turn the loose ends into degree $1$ vertices, so that the arrangement becomes a standard embedded graph. The linear code determines this embedded graph and its rotation system, exactly as in Section~\ref{sec:codes}. Its outer face is determined as well: it is the unique face incident with the degree $1$ vertices. Thus the plane embedding is determined, and restoring the disjoint tails outside the disk recovers the arrangement up to isomorphism. Finally, the proof of Lemma~\ref{lem:triplesdo} applies without change, so a linear code of size at least $3$ is determined by its subcodes of size $3$.

\begin{proof}[Proof of the upper bound in Theorem~\ref{thm:main} for $|\LL_{k,m}|$]
Fix any $b>64D^3$. A fixed isomorphism class of $k$-arrangements of two pseudolines gives at most $16$ linear codes. Indeed, there are two possible orientations of the plane, two labellings of the pseudolines, and two choices of direction on each pseudoline. The number of these isomorphism classes is $|L(k,2)|$, which is less than $D^k$ by Lemma~\ref{lem:meanbound}, with $D$ as in the statement of the lemma. Thus $|\PPi^{\mathrm L}_{k,2}|<16 D^k$.

Once the three subcodes of size $2$ of a linear code of size $3$ are fixed, each key is obtained by merging two prescribed strings of length $k$, while preserving the order within each string. Therefore, as in the proof of Lemma~\ref{lem:third},
\[
|\PPi^{\mathrm L}_{k,3}|\le |\PPi^{\mathrm L}_{k,2}|^3\binom{2k}{k}^{\!3}<4096\cdot\bigl(64D^3\bigr)^k<b^k
\]
for all sufficiently large $k$.

For $m\ge3$, we repeat the Ramsey argument used for pseudocircles, now using linear codes. If two arrangements give the same code of size $3$ in that argument, their full linear codes are equal, and hence the arrangements are isomorphic. Therefore for all sufficiently large $k$ we have
\[
|\LL_{k,m}|\le |\PPi^{\mathrm L}_{k,3}|<b^k.
\]
\end{proof}

\section{Concluding remarks}\label{sec:concluding}

We conclude with three remarks. In Section~\ref{sub:regpat} we discuss the relationship between regular path systems and $k$-arrangements of pseudolines. In Section~\ref{sub:regmaps} we point out that the constant in Lemma~\ref{lem:third} can be improved by exploiting the relationship between $k$-arrangements of pseudocircles and $4$-regular planar maps. Finally, in Section~\ref{sub:pacmea} we give a more detailed explanation of the difficulty of finding an explicit characterization of all $k$-arrangements (of pseudolines or pseudocircles) that are packable.

\subsection{Regular path systems and $x$-monotone $k$-arrangements}\label{sub:regpat}

Dobbins, Holmsen, and Hubard~\cite{DHH} study arrangements of pseudolines in the closely related setting of {\em regular path systems}. As we illustrate in Figure~\ref{fig:dhh1}, a {\em path system} consists of a finite collection of paths in a vertical strip, where each path is the graph of a continuous function $[0,1]\to\mathbb R$. Any two paths meet in a finite number of points, at which they cross, and no three paths meet at a common point. The paths are labelled from bottom to top according to the order of their left endpoints. 

For three paths labelled $i<j<\ell$, their {\em signature} is the word on $x,y,z$ that records their crossings from left to right, using $x$ for crossings of $i$ and $j$, $y$ for crossings of $i$ and $\ell$, and $z$ for crossings of $j$ and $\ell$. We refer the reader to Figure~\ref{fig:dhh1} and its caption for an illustration.

\begin{figure}[ht!]
\centering
\def\ta#1{{\Scale[1.4]{#1}}}
\scalebox{0.7}{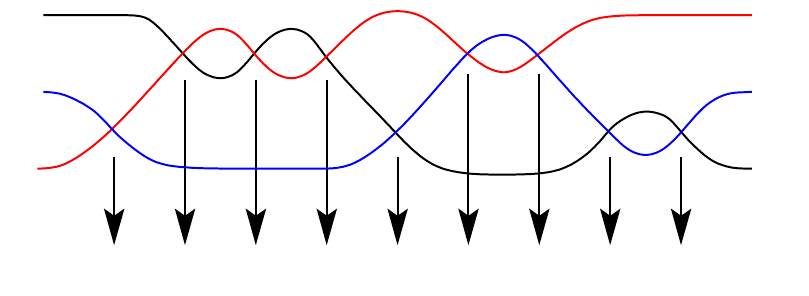}
\caption{The system depicted in the figure has signature $\tau=xy^3zx^2z^2$.}
\label{fig:dhh1}
\end{figure}

A system of at least four paths is {\em regular} if all its three-path subsystems have the same signature. This common signature is the {\em signature} of the system. A signature $\tau$ is {\em extendable} if it is the signature of a regular system of four paths. Dobbins, Holmsen, and Hubard characterized the extendable signatures, and proved that if $\tau$ is extendable, then for every $n\ge 4$ there is a regular system of size $n$ with signature $\tau$ (see~\cite[Section 5]{DHH}).

The signature of a regular system determines the order in which each path meets the others~\cite[see Lemmas~3.12,~4.8 and~4.9]{DHH}, and hence determines the system up to isomorphism. Therefore, if we extend each path infinitely in both directions, a regular system of size $m$ in which every pair of paths crosses exactly $k$ times gives a packable $k$-arrangement of $m$ $x$-monotone pseudolines.

Moreover, it can be verified that the characterization by Dobbins, Holmsen and Hubard~\cite{DHH} actually identifies all packable $k$-arrangements of $m$ $x$-monotone pseudolines. From their characterization one can obtain a lower bound on the number of these packable arrangements, which thus provides a lower bound on $|\LL_{k,m}|$. According to our calculations, as one would expect this bound is slightly better than the simple bound we obtained in Section~\ref{sec:lower}, but is still far from the upper bound we proved in Section~\ref{sec:uplin}.

\subsection{Arrangements of pseudocircles and $4$-regular planar maps}\label{sub:regmaps}

We now note that the upper bound for $|\CC_{k,m}|$ can be improved using a result from the enumerative theory of planar maps. 

Let $B$ be a $k$-arrangement of $n\ge 2$ pseudocircles. As we did in Section~\ref{sec:codes}, we may regard $B$ as a $4$-regular embedded graph (that is, plane map) $G_B$. It is straightforward to see that $G_B$ is $2$-connected and loopless. In the case $n=3$ it has $3k$ vertices and $6k$ edges, and so by Euler's formula it has $3k+2$ faces. Ren, Liu and Li~\cite{RLL} proved that the number of rooted $2$-connected loopless $4$-regular planar maps with $f+1$ faces is asymptotically $c\cdot f^{-5/2}(196/27)^f$ for some constant $c$. Since two arrangements of pseudocircles are isomorphic precisely when their maps are homeomorphic, and since an unrooted map has at least one rooting, we conclude that the set $C(k,3)$ of isomorphism classes of $k$-arrangements of three pseudocircles satisfies that
\[
|C(k,3)|\ \le\ c_0 \cdot\Bigl(\frac{196}{27}\Bigr)^{3k}
\]
for some absolute constant $c_0$. 

Similarly as in the proof of Lemma~\ref{lem:third}, each isomorphism class of $k$-arran\-gements of three pseudocircles yields at most $2^3\cdot 2\cdot 3!\cdot(2k)^3$ codes: indeed, we must specify the bounded side of each of the three pseudocircles, an orientation, a labelling, and a startpoint on each pseudocircle. Therefore $|\PPi_{k,3}|$ is at most a polynomial in $k$ times $|C(k,3)|$. Using this in~\eqref{eq:suff1} instead of Lemma~\ref{lem:third} shows that the upper bound for $|\CC_{k,m}|$ in Theorem~\ref{thm:main} holds with any $b>(196/27)^3\approx 382.54$, whereas the constant $64D^3$ obtained in Section~\ref{sec:upcirc} is considerably larger.

We chose to give the argument of Section~\ref{sec:upcirc} as the proof for two reasons. First, our argument is elementary and self-contained, whereas the asymptotics given by Ren, Liu and Li~\cite{RLL} rely on a singularity analysis of an algebraic generating function. Second, our proof treats pseudolines and pseudocircles uniformly, while the map-theoretic argument does not immediately apply to pseudolines.

\subsection{Packable arrangements and meanders}\label{sub:pacmea}

We finally discuss the reason behind our passing remark in Section~\ref{sec:intro} to the effect that a full answer to Question (I) seems hopeless at the moment (for pseudolines, although analogous complications apply for pseudocircles). The reason, as we are about to argue, is that a full classification of packable $k$-arrangements is at least as complicated as a full classification of meanders, a classical long-standing open question.

Recall from Section~\ref{sec:upperbound} that $L(k,2)$ denotes the set of isomorphism classes of $k$-arrangements of two pseudolines. Straightening one of the two pseudolines into a horizontal line turns such an arrangement into an open meander of order $k$, and conversely every open meander of order $k$ arises in this way. The choices involved are which of the two pseudolines is straightened and how it is straightened: two straightenings of the same pseudoline differ by a self-homeomorphism of the plane taking the line to itself, and there are four such classes of these up to equivalence. Therefore each element of $L(k,2)$ corresponds to at least one and at most eight open meanders of order $k$. In this sense $L(k,2)$ {\em is} the set of open meanders of order $k$.

Let $R_m$ denote the cyclic $1$-arrangement of size $m$, let $A\in L(k,2)$, and take a closed disk $\Delta$ containing all the crossings of $A$. We let $R_m(A)$ be the arrangement obtained from $R_m$ by replacing a small disk around each crossing with a copy of $A\cap\Delta$.

\begin{observation}\label{obs:meanpack}
Let $k$ be odd and let $m\ge2$. Then every $A\in L(k,2)$ induces a packable $k$-arrangement $R_m(A)$ of $m$ pseudolines, and the map $A\mapsto R_m(A)$ is injective. In particular $\LL_{k,2}=L(k,2)$.
\end{observation}

\def\bpl{{\Scale[3]{}}}
\def\brep{{\Scale[1.4]{\text{\rm replace each crossing}}}}
\def\abrep{{\Scale[1.4]{\text{\rm \hglue 0.9 cm with } A}}}
\def\aerretres{{\Scale[1.8]{R_3(A)}}}
\def\erretres{{\Scale[1.8]{R_3}}}
\def\arra{{\Scale[1.8]{A}}}
\begin{figure}[ht!]
\centering
\scalebox{0.58}{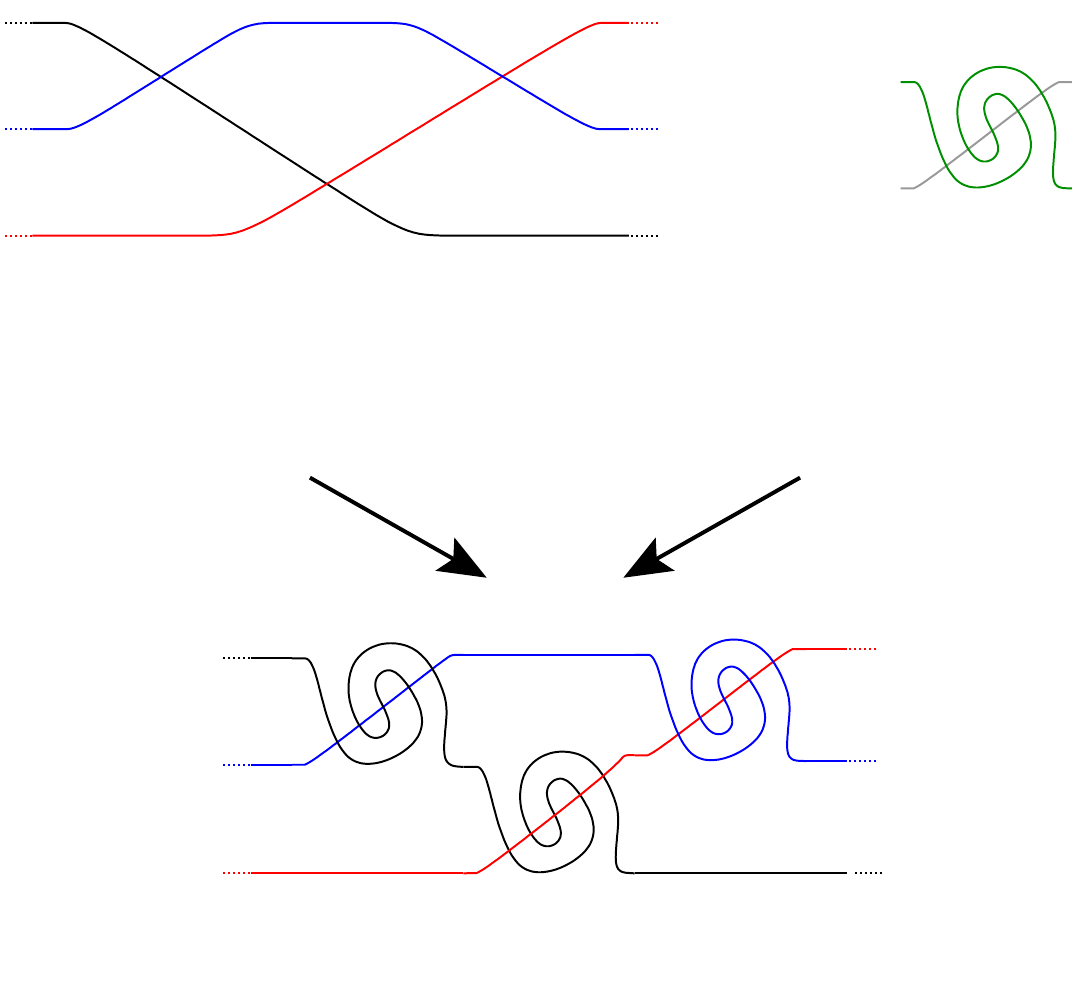}
\caption{Illustration of Observation~\ref{obs:meanpack}. We replace each crossing in the cyclic $1$-arrangement $R_3$ of size $3$ with a copy of the $5$-arrangement $A$ of two pseudolines. As a result we obtain a packable $5$-arrangement $R_3(A)$ of $3$ pseudolines.}
\label{fig:660}
\end{figure}

\begin{proof}
Since $k$ is odd, the two arcs of $A\cap\Delta$ exchange sides inside $\Delta$, and so this substitution is consistent with the crossing it replaces; each pair of pseudolines of $R_m(A)$ crosses exactly $k$ times. As every subarrangement of size $m$ of $R_n(A)$ is isomorphic to $R_m(A)$ for every $n\ge m$, the arrangement $R_m(A)$ is packable. 

Finally, any two pseudolines of $R_m(A)$ induce $R_2(A)=A$. This implies that the map $A\mapsto R_m(A)$ is injective, and it also implies that $\LL_{k,2}=L(k,2)$.
\end{proof}

It follows from this observation that for every odd $k$ and every $m \ge 2$ a classification of the packable $k$-arrangements of $m$ pseudolines would contain a classification of the open meanders of order $k$. Even the {\em enumeration} of meanders is a classical and difficult problem (we refer the reader again to the work by Albert and Paterson~\cite{albert} or to the survey by La~Croix~\cite{lacroix}), and this is our reason for not attempting to answer Question~(I): the packable $k$-arrangements are at least as complicated as the meanders. The analogous relation between $k$-arrangements of two pseudocircles and closed meanders of order $k/2$, illustrated in Figure~\ref{fig:closmea}, suggests the same obstruction for pseudocircles and for even $k$ in the pseudoline case.

We note in passing that combining Observation~\ref{obs:meanpack} with the correspondence between $L(k,2)$ and open meanders described above we obtain that for every odd $k$ and every $m\ge 2$
\[
|\LL_{k,m}|\ \ge\ |L(k,2)|\ \ge\ \frac{m_k}{8},
\]
where $m_k$ denotes the number of open meanders of order $k$. 

Since $m_k$ grows exponentially with a base strictly larger than $3$ (this follows from the bounds given by Albert and Paterson~\cite{albert}, together with the well-known correspondence between open meanders of order $2r-1$ and closed meanders of order $r$), for all sufficiently large odd $k$ this is a better lower bound than the one given by the construction in Section~\ref{sec:lower}. Moreover, this bound is also better than the one obtained using the characterization by Dobbins, Holmsen, and Hubard, which according to our estimates is of the order $(3+o(1))^k$. 

We do not further pursue this improvement, first of all because the construction in Section~\ref{sec:lower} is more elementary and applies uniformly to pseudolines and to pseudocircles. Finally, even this slightly improved bound is still far from the upper bounds we proved in Section~\ref{sec:uplin} or Section~\ref{sub:regmaps}. In our upper bounds we estimated the number of \emph{all} $k$-arrangements of three curves, and we did not use the packability property in any substantial way. Therefore we expect that significant improvements of the upper bound on \emph{packable} $k$-arrangements of three pseudolines or pseudocircles are still possible.

\section*{Acknowledgements}

The first author was supported by Project 23-04949X of the Czech Science Foundation (GA\v{C}R). The second and third authors were supported by CONACYT under grant 191952.

\bibliographystyle{amsplain}
\bibliography{refs}

\end{document}